\documentclass[final]{siamltex}
\usepackage{texdraw,amsmath,amssymb,bm,cite,graphicx}
\newcommand{\beq}{\begin{equation}}
\newcommand{\eeq}{\end{equation}}
\newcommand{\bea}{\begin{eqnarray}}
\newcommand{\eea}{\end{eqnarray}}
\newcommand{\bef}{\begin{figure}}
\newcommand{\eef}{\end{figure}}
\newcommand{\bsc}{\begin{scriptstyle}}
\newcommand{\esc}{\end{scriptstyle}}
\newcommand{\bd}{\begin{displaymath}}
\newcommand{\ed}{\end{displaymath}}
\newcommand{\nn}{\nonumber}
\newcommand{\nnl}{\nonumber \\}
\newcommand{\fig}[1]{Fig.\ \ref{#1}}

\def\bmat{\left[ \begin{array}}
\def\emat{\end{array} \right]}
\newcommand{\defn}{\stackrel{\triangle}{=}}

\title{A Contraction Analysis of the Convergence
of Risk-Sensitive Filters}

\author{Bernard C. Levy 
\thanks{B. Levy is with the Department of Electrical and Computer
Engineering, 1 Shields Avenue, University of California, Davis,
CA 95616 (email: bclevy@ucdavis.edu).}
\and Mattia Zorzi
\thanks{M. Zorzi is  with the Department of Electrical
Engineering and Computer Science, University of Liege, Institut
Montefiore B28, 4000 Liege, Belgium (email: mzorzi@ulg.ac.be).}}

\begin{document}

\pagestyle{myheadings}
\thispagestyle{plain}
\markboth{B.~C. LEVY AND M. ZORZI}{CONVERGENCE OF
RISK-SENSITIVE FILTERS}

\maketitle

\begin{abstract} 
A contraction analysis of risk-sensitive Riccati equations
is proposed. When the state-space model is reachable
and observable, a block-update implementation of the
risk-sensitive filter is used to show that the $N$-fold
composition of the Riccati map is strictly contractive with
respect to the Riemannian metric of positive definite 
matrices, when $N$ is larger than the number of
states. The range of values of the risk-sensitivity 
parameter for which the map remains contractive can be
estimated a priori. It is also found that a second
condition must be imposed on the risk-sensitivity parameter
and on the initial error variance to ensure that the solution
of the risk-sensitive Riccati equation remains positive
definite at all times. The two conditions obtained
can be viewed as extending to the multivariable case
an earlier analysis of Whittle for the scalar case.
\end{abstract}



\section{Introduction}
\label{sec:intro}

Starting with Kalman and Bucy's paper \cite{KB}, the convergence
of the Kalman filter has been examined in detail, and it soon 
became clear that if the state-space model is stabilizable and
detectable, the filter is asymptotically stable and the
error covariance converges to the unique non-negative definite 
solution of a matching algebraic Riccati equation. However
the classical Kalman filter convergence analysis 
\cite{AM,KSH} is rather intricate and involves several steps,
including first showing that the error covariance is upper
bounded, next proving that with a zero initial value,
it is monotone increasing, so it has a limit, and then 
establishing that the corresponding filter is stable and that 
the limit is the same for all initial covariances. In 1993,
Bougerol \cite{Bou} proposed a more direct convergence proof
based on establishing that the discrete-time Riccati iteration 
is a contraction for the Riemmanian metric associated to
the cone of positive definite matrices. Although this result
attracted initially little notice in the systems and control
community, this approach was adopted by several researchers
\cite{LaL1,LaL2,LeL,LYD} to study the convergence of a number 
of nonlinear matrix iterations. We will use this viewpoint here 
to analyze the convergence of risk-sensitive estimation filters. 
Unlike the Kalman filter which minimizes the mean square
estimation error, risk-sensitive filters \cite{SDJ,Whi} minimize
the expected value of an exponential of quadratic error index, 
which ensures a higher degree of robustness \cite{HSK1,HS} against 
modelling errors. Unfortunately, in spite of extensive studies 
on risk-sensitive and related $H^{\infty}$ filters, results
concerning their convergence remain fragmentary \cite[Chap. 9]{Whi},
\cite[Sec. 14.6]{HSK1}, \cite{BS}. In particular, one question
that remains unresolved is whether there exists an a-priori
upper bound on the risk-sensitivity parameter ensuring the
convergence of the solution of the risk-sensitivite Riccati 
equation to a positive definite solution associated to
a stable filter.

The contraction anaysis presented in this paper relies
on a block implementation of Kalman (risk-neutral) 
and risk-sensitive filters. When the system is
reachable and observable and the block length $N$
exceeds the number of states, it is shown that in
the risk-neutral case, the Riccati equation corresponding 
to the block filter is strictly contractive, which allows us to
conclude that the Riccati equation of Kalman filtering 
has a unique positive definite fixed point. This 
analysis is equivalent to the derivation of \cite{Bou}
which relied on showing that the $N$-fold composition
of the Hamiltonian map associated to the 
risk-neutral Riccati operator is strictly contractive. 
However, it has the advantage that it can be extended
easily to the risk-sensitive case by using the Krein-space 
formulation of risk-sensitive and $H^{\infty}$ filtering 
developed in \cite{HSK2,HSK3}. With this approach, it is
is shown that the $N$-block risk-sensitive Riccati
equation remains strictly contractive as long as
a corresponding observability Gramian is positive
definite. This Gramian is shown to be a monotone
decreasing function of  the risk-sensitivity parameter
$\theta$  with respect to the partial order of non-negative
definite matrices. Accordingly, it is possible to identify
a priori a range $[0,\tau_N)$ of values of the 
risk-sensitivity  parameter $\theta$ for which the block
Riccati equation is strictly contractive. This result 
is used to show that the risk-sensitive Riccati
equation has a unique positive definite fixed-point, but 
because the image of the cone ${\cal P}$ of positive definite 
matrices under the risk-sensitive Riccati map is not entirely 
contained in ${\cal P}$ a second condition must be placed on
$\theta$ and the initial variance $P_0$ of the filter
to ensure that the evolution of the risk-sensitive
Riccati equation stays in ${\cal P}$. The two conditions
obtained can be viewed as extensions to the multivariable
case of those presented in \cite[Chap. 9]{Whi} for scalar 
risk-sensitive Riccati equations.

The paper is organized as follows. The properties of
the Riemann distance for positive definite matrices and
of contraction mappings are reviewed in Section 
\ref{sec:riemann}. The block-update filtering interpretation
of the $N$-fold Riccati equation of Kalman filtering
is described in Section \ref{sec:block} and is extended
to the risk-sensitive case in Section \ref{sec:contrisk}.
This formulation is used to estimate the range of
values of the risk-sensitivity parameter for which the
risk-sensitive Riccati equation is contractive. A
second condition on the risk-sensitivity parameter
and initial condition ensuring that the solution
of the Riccati equation remains positive is obtained
in Section \ref{sec:posit}. An illustrative example is
studied in Section \ref{sec:examp} and conclusions 
as well as a possible extension are presented in
Section \ref{sec:conc}.

\section{Riemann distance and contraction mappings}
\label{sec:riemann}

Let ${\cal P}$ denote the cone of positive definite symmetric
matrices of dimension $n$, and let $\bar{\cal P}$ denote
the cone of non-negative definite matrices forming the closure
of ${\cal P}$. If $P$ is an element of ${\cal P}$
with eigendecomposition
\beq
P = U \Lambda U^T \label{2.1}
\eeq
where $U$ is an orthogonal matrix formed by normalized
eigenvectors of $P$ and $\Lambda = \mbox{diag} \, \{
\lambda_1 , \ldots , \lambda_n \}$ is the diagonal 
eigenvalue matrix of $P$, the symmetric positive square-root
of $P$ is defined as
\[
P^{1/2} = U \Lambda^{1/2} U^T
\]
where $\Lambda^{1/2}$ is diagonal, with entries $\lambda_i^{1/2}$
for $1 \leq i \leq n$. Similarly, the logarithm of $P$ is
the symmetric, not necessarily positive definite, matrix
specified by
\[
\log(P) = U \log(\Lambda ) U^T \: , 
\]
where $\log(\Lambda)$ is diagonal with entries $\log(\lambda_i)$
for $1 \leq i \leq n$. Let $P$ and $Q$ be two positive definite
matrices of ${\cal P}$. Then $P^{-1}Q$ is similar to $P^{-1/2}QP^{-1/2}$,
so they have the same eigenvalues, and $P^{-1/2}QP^{-1/2}$ is
positive definite. Let $s_1 \geq s_2 \geq \ldots \geq s_n >0$
denote the eigenvalues of $P^{-1}Q$ sorted in decreasing order.
The Riemann distance between $P$ and $Q$ is defined as 
\beq
d(P,Q) = ||\log(P^{-1/2}QP^{-1/2})|| = \Big( \sum_{i=1}^n 
( \log (s_i))^2 \Big)^{1/2} \: , \label{2.2}
\eeq
where $||.||$ denotes the matrix Frobenius norm. In addition
to having all the traditional properties of a distance, $d$
has the feature that it is invariant under matrix inversion
and congruence transformations. Specifically, if $M$ denotes
an arbitrary real invertible matrix of dimension $n$,
\beq
d(P,Q) = d(P^{-1},Q^{-1}) = d(MPM^T,MQM^T) \: . \label{2.3} 
\eeq
Furthermore, it was also shown by Bougerol \cite{Bou} that
the translation of ${\cal P}$ by a non-negative definite
symmetric matrix $S$ is a non-expansive map. Specifically, 
\beq
d(P+S,Q+S) \leq  \frac{\alpha}{\alpha+\beta} d(P,Q) \label{2.4}
\eeq
where $\alpha = \max (\lambda_1(P),\lambda_1(Q))$ and
$\beta = \lambda_n (S)$. In these definitions, it is
assumed that the eigenvalues of $P$, $Q$ and $S$ are
sorted in decreasing order, so that $\lambda_1 (P)$
is the largest eigenvalue of $P$, i.e., its spectral
norm, and $\lambda_n (S)$ is the smallest eigenvalue of 
$S$. 

Recall that if $f(\cdot)$ is an arbitrary mapping of
${\cal P}$, $f$ is non-expansive if
\[
d(f(P),f(Q)) \leq d(P,Q) \: ,
\]
and strictly contractive if
\[
d(f(P),f(Q)) \leq c d(P,Q) 
\]
with $0 \leq c < 1$. The least contraction coefficient or
Lipschitz constant of a non-expansive mapping $f$ is defined as
\beq
c(f) = \sup_{P,Q \in {\cal P}, P \neq Q} \frac{d(f(P),f(Q)}{d(P,Q)}
\: . \label{2.5}
\eeq
Clearly, if $f$ and $g$ denote two non-expansive mappings, 
the contraction coefficient $c(f \circ g)$ of the composition 
of $f$ and $g$ satisfies $c (f \circ g) \leq c(f) c(g)$, so if
at least one of the two maps is strictly contractive, the
composition is also strictly contractive. From 
inequality (\ref{2.4}), we deduce that if $\tau_S (P) = P+S$
denotes the translation by a positive definite matrix
$S$, $\tau_S$ is non-expansive, but the bound (\ref{2.4}) does 
not allow us to conclude that $c(\tau_S) < 1$, since when the largest 
eigenvalue of either $P$ or $Q$ goes to infinity, the constant 
$\alpha/(\alpha+\beta)$ tends to one.

The key result that will be used in this paper is that if 
$f$ is a strict contraction of ${\cal P}$ for the distance $d$, 
by the Banach fixed point theorem \cite[p. 244]{AE}, there 
exists a unique fixed point $P$ of $f$ in $\bar{\cal P}$ 
satisfying $P=f(P)$. Furthermore this fixed point can be 
evaluated by performing the iteration $P_{n+1} = f(P_n)$
starting from any initial point $P_0$ of ${\cal P}$. Also
if the $N$-fold composition $f^N$ of a non-expansive map
$f$ is strictly contractive, then $f$ has a unique fixed 
point. We will consider in particular the Riccati-type map
over ${\cal P}$ defined by
\beq
f(P) = M[P^{-1} + \Omega ]^{-1}M^T + W \label{2.6}
\eeq
where $P$, $\Omega$ and $W$ are symmetric real positive definite 
matrices and $M$ is a square real, but not necessarily invertible,
matrix. For this mapping the following result was established
in \cite[Th. 4.4]{LeL}.
\vskip 2ex

\begin{lemma} 
\label{lem1}
$f$ is a strict contraction with 
\beq
c(f) \leq \frac{\lambda_1 (M\Omega^{-1}M^T)}{
\lambda_n (W) + \lambda_1(M\Omega^{-1}M^T)}
\: , \label{2.7}
\eeq
where we use again the convention that the eigenvalues
of positive definite matrices are sorted in decreasing
order. 
\end{lemma}
\vskip 2ex

Note that although the results presented in this paper
use the Riemann distance over ${\cal P}$, other
metrics such as the Thompson metric
\beq
d_T (P,Q) = \max ( \lambda_1 (\log(P^{-1/2}QP^{-1/2})),
\lambda_1 (\log(Q^{-1/2} P Q^{-1/2}))) \: . 
\label{2.8}
\eeq 
have similar properties \cite{Bha,ISY}, and the Thompson
metric is in fact often used to analyze the convergence
of nonlinear matrix iterations \cite{LaL2,LYD}.

\section{Block update filter}
\label{sec:block}

Consider a Gauss-Markov state space model 
\bea
x_{t+1} =A x_t + B u_t \label{3.1} \\
y_t = C x_t + v_t \: , \label{3.2}
\eea
where the state $x_t \in \mathbb{R}^n$, the process 
noise $u_t \in \mathbb{R}^m$ and the observation noise
$v_t \in \mathbb{R}^p$. The noises $u_t$ and $v_t$
are assumed to be independent zero-mean WGN processes
with normalized covariance matrices, so
\[
E \Big[ \bmat {c}
u_t \\ v_t
\emat \bmat {cc}
u_s^T & v_s^T
\emat \Big] = \bmat {cc} 
I_m & 0\\
0 & I_p
\emat \delta_{t-s} \: , 
\]
where 
\[
\delta_t = \left \{ \begin{array} {cc}
1 & t=0 \\
0 & t \neq 0 
\end{array} \right.
\]
denotes the Kronecker delta function. The initial
state vector $x_0$ is assumed independent of noises
$u_t$ and $v_t$ and ${\cal N} (\hat{x}_0,P_0)$ 
distributed. Since we are interested in the asymptotic behavior 
of Kalman and risk-sensitive filters, the matrices $A$, $B$, and
$C$ specifying the state-space model are assumed to be constant.
Then if ${\cal Y}_{t-1}$ denotes the sigma field generated 
by observations $y(s)$ for $0 \leq s \leq t-1$, the 
least-squares estimate $\hat{x}_t = E[x_t|{\cal Y}_{t-1}]$
depends linearly on the observations and can be  evaluated
recursively by the predicted form of the Kalman filter
specified by
\beq   
\hat{x}_{t+1} = A \hat{x}_t + K_t \nu_t
\: , \label{3.3}
\eeq
where the innovations process
\beq
\nu_t \defn y_t - C\hat{x}_t \: . \label{3.4}
\eeq
In (\ref{3.3}), the Kalman gain matrix
\beq
K_t = AP_tC^T(R_t^{\nu})^{-1} \: , \label{3.5}
\eeq
where 
\beq
R_t^{\nu} = E[\nu_t \nu_t^T] = CP_tC^T +I_p
\label{3.6}
\eeq
represents the variance of the innovations process, and
if $\tilde{x}_t = x_t - \hat{x}_t$ denotes the
state prediction error, its variance matrix $P_t =
E[\tilde{x}_t \tilde{x}_t^T]$ obeys the Riccati
equation
\beq
P_{t+1}= r(P_t) \defn A[ P_t^{-1} + C^TC ]^{-1} A^T + BB^T
\label{3.7}
\eeq
with initial condition $P_0$. This equation can also be rewritten
in the equivalent form \cite[p. 325]{KSH} 
\beq
P_{t+1} = (A-K_tC)P_t(A-K_t C)^T + BB^T + K_t K_t^T 
\: , \label{3.8}
\eeq
which will be used later in our analysis.

The Riccati mapping $r(P)$ specified by (\ref{3.7}) has
the form (\ref{2.6}). Unfortunately the matrices $C^TC$
and $BB^T$ are not necessarily invertible, so Lemma \ref{lem1}
is not directly applicable. Under the assumption that
the pairs $(A,B)$ and $(C,A)$ are reachable and observable,
respectively, Bougerol \cite{Bou} was able to show that the 
$n$-fold map $r^n$ is a strict contraction. This was achieved by
considering the $n$-fold composition of the symplectic Hamiltonian
mapping associated to $r$ (see \cite{LaL2} for a study 
of the contraction properties of symplectic Hamiltonian 
mappings). We present below an equivalent derivation of
of Bougerol's result which relies on a block update 
implementation of the Kalman filter.

The starting point is the observation that since $x_t$ is
Gauss-Markov, the downsampled process $x_k^d = x_{kN}$ with
$N$ integer is also Gauss-Markov with state-space model
\bea
x_{k+1}^d &=& A^N x_k^d + {\cal R}_N {\bf u}_k^N \label{3.9} \\ 
{\bf y}_k^N &=& {\cal O}_N x_k^d + {\bf v}_k^N + {\cal H}_N {\bf u}_k^N 
\label{3.10}
\eea
where
\bea
{\bf u}_k^N &=& \bmat {cccc}
u_{kN+N-1}^T & u_{kN+N-2}^T & \ldots & u_{kN}^T
\emat^T \nnl
{\bf y}_k^N &=& \bmat {cccc}
y_{kN+N-1}^T & y_{kN+N-2}^T & \ldots & y_{kN}^T
\emat^T \nnl
{\bf v}_k^N &=& \bmat {cccc}
v_{kN+N-1}^T & v_{kN+N-2}^T & \ldots & v_{kN}^T
\emat^T \:. \nn 
\eea
In the model (\ref{3.9})--(\ref{3.10})
\bea
{\cal R}_N &=& \bmat {cccc}
B & AB & \ldots & A^{N-1}B
\emat \nnl 
{\cal O}_N &=& \bmat {cccc} 
(CA^{N-1})^T & \ldots & (CA)^T & C^T
\emat^T \nn
\eea
denote respectively the $N$-block reachability and observability
matrices of system (\ref{3.1})--(\ref{3.2}, where the blocks 
forming ${\cal O}_N$ are written from bottom to top instead 
of the usual top to bottom convention. If the pairs
$(A,B)$ and $(C,A)$ are reachable and observable, ${\cal R}_N$ 
and ${\cal O}_N$ have full rank for $N \geq n$. In (\ref{3.10}),
if
\[
H_t = \left \{ \begin{array} {cc}
CA^{t-1}B & t \geq 1 \nnl
0 & \mbox{otherwise }
\end{array} \right. 
\]
denotes the impulse response representing the response of 
output $y_t$ in (\ref{3.2}) to the process noise $u_t$ input
in (\ref{3.1}), ${\cal H}_N$ is the $Np \times Nm$ block
Toeplitz matrix defined by
\[
{\cal H}_N = \bmat {cccccc}
0 & H_1 & H_2 & \cdots & H_{N-2} &H_{N-1}\\ 
0 & 0   & H_1 &  H_2 & \cdots & H_{N-2}\\ 
0 & 0   &  0 &  H_1 &  \cdots & H_{N-3}\\ 
\vdots & \vdots &  \vdots &  & \vdots & \vdots\\
0 & 0   & 0 & \cdots & \cdots & H_1 \\
0 & 0   & 0 & \cdots & \cdots & 0 
\emat \: .
\]
Note however that the noise vectors ${\bf u}_k^N$ and
${\bf w}_k^N \defn {\bf v}_k^N +{\cal H}_N {\bf u}_k^N$ 
are correlated since
\[
E \Big[ \bmat {c}
{\bf u}_k^N \\ {\bf w}_k^N 
\emat \bmat {cc}
{\bf u}_{\ell}^{N T} & {\bf w}_{\ell}^{N T}
\emat
= \bmat {cc}
I_{Nm} & {\cal H}_N^T \\
{\cal H}_N & I_{Np} + {\cal H}_N {\cal H}_N^T
\emat \delta_{k-\ell} \: . 
\]
This correlation can be removed by noting that the estimate
of ${\bf u}_k^N$ given ${\bf w}_k^N$ takes the form
\[
\hat{u}_k^N = {\cal G}_N {\bf w}_k^N
\]
where
\[
{\cal G}_N ={\cal H}_N^T (I+ {\cal H}_N{\cal H}_N^T)^{-1} \:.
\]
Then by premultiplying the observation equation (\ref{3.10})
by ${\cal R}_N {\cal G}_N$ and subtracting it from (\ref{3.9})
we obtain the new downsampled state dynamics
\beq
x_{k+1}^d = \alpha_N x_k^d + {\cal R}_N \tilde{\bf u}_k^N
+{\cal R}_N {\cal G}_N {\bf y}_k^N
\label{3.11}
\eeq
with
\[
\alpha_N \defn A^N - {\cal R}_N {\cal G}_N {\cal O}_N \: , 
\]
where the zero mean white Gaussian noise $\tilde{\bf u}_k^N 
= {\bf u}_k^N - \hat{\bf u}_k^N$ is now uncorrelated with
observation noise ${\bf w}_k^N$, and has the invertible
variance matrix
\bea
{\cal Q}_N &=& I_{Nm}-{\cal H}_N^T [ I_{Np} + {\cal H}_N
{\cal H}_N^T ]^{-1} {\cal H}_N  \nnl
&=& [ I_{Nm} + {\cal H}_N^T {\cal H}_N ]^{-1} \: . \nn
\eea

The Kalman filter corresponding to the downsampled state-space
model (\ref{3.10})--(\ref{3.11}) can be interpreted as a block
update filter, where the state estimate is updated only after
a block of $N$ observations has been collected. The Riccati
equation corresponding to this Kalman filter is then given by
\beq
P_{k+1}^d = r_d (P_k^d) = \alpha_N [ (P_k^d)^{-1} + \Omega_N ]^{-1}
\alpha_N^T + W_N \: , \label{3.12}
\eeq
where the $n \times n$ symmetric real matrices 
\bea
\Omega_N &\defn& {\cal O}_N^T [I + {\cal H}_N{\cal H}_N^T]^{-1} {\cal O}_N 
\label{3.13} \\
\nnl
W_N &\defn& {\cal R}_N [I+ {\cal H}_N^T {\cal H}_N]^{-1} {\cal R}_N^T 
\label{3.14}
\eea
are positive definite for $N \geq n$ whenever the pairs
$(C,A)$ and $(A,B)$ are observable and reachable, respectively.
In fact, $\Omega_N$ and $W_N$ can be viewed as observability 
and reachability Wronskians for the state-space model
(\ref{3.1})--(\ref{3.2}).

From Lemma \ref{lem1}, we can therefore conclude that $r_d(\cdot)$ is
a strict contraction. However, since $P_k^d$ is the variance
matrix of the one-step ahead prediction error for state
$x_{kN}$, $r_d$ coincides with the $N$-fold composition
$r^N$ of Riccati map $r$, which must have therefore a unique
fixed point $P$ in ${\cal P}$. This establishes the
following classical Kalman filter convergence result
\cite{AM,KSH}.
\vskip 2ex

\begin{theorem}
\label{theo1}
If in system (\ref{3.1})-(\ref{3.2}) the pairs $(A,B)$ and $(C,A)$ are 
reachable and observable, respectively, the algebraic Riccati equation
$P=r(P)$ admits a unique positive definite solution, and as
$t$ tends to infinity, for any positive definite initial
condition $P_0$, $P_t$ tends to $P$ as $t$ tends to infinity,
and the Kalman gain matrix $K_t$ tends to 
\[
K =  APC^T [CPC^T +I]^{-1}
\]
which has the property that the matrix $A-KC$ is stable.
\end{theorem}
\vskip 2ex

Given the fixed point $P >0$, the stability of $A-KC$ 
is obtained by applying the Lyapunov stability theorem to
equation
\beq
P = (A-KC)P(A-KC)^T +BB^T +KK^T \label{3.15} 
\eeq
(see \cite[p. 80]{AM}). 

One unsatisfactory aspect of the contraction approach
to the derivation of Theorem \ref{theo1} is its requirement that
the system should be reachable and observable, instead
of the weaker stabilizability and detectability conditions 
required by conventional Kalman filter convergence proofs
\cite{AM,KSH}. The stronger conditions conditions are needed 
to ensure that the Riccati evolution takes place entirely in 
the cone of positive definite matrices. On the other hand,
if the system is reachable and observable, the limit
$P$ is guaranteed to be positive definite, instead
of just nonnegative definite under the usual assumptions.
Finally, note that the block update implementation of the 
Kalman filter which was used here to show that $r_d=r^N$
is a strict contraction is equivalent to Bougerol's derivation
in \cite{Bou}, but as shown below it can be extended
more easily to the risk-sensitive case.

\section{Contraction property of the risk-sensitive 
Riccati equation}
\label{sec:contrisk}

For the state-space model (\ref{3.1})--(\ref{3.2}), the 
risk-sensitive estimate $\hat{x}_t$ solves the exponential 
quadratic minimization problem \cite{Whi,SFB}
\beq
\hat{x}_t = \mbox{arg} \, \min_{\xi \in \mathbb{R}^n} 
\frac{1}{\theta} \log \Big( E[ \exp \big( \frac{\theta}{2}
||D(x_t-\xi)||^2 \big) | {\cal Y}_{t-1}] \Big) 
\: , \label{4.1}
\eeq
where $D \in {\mathbb R}^{q \times n}$ with $q \leq n$
is assumed to have full row rank, and $||z|| = (z^T z)^{1/2}$
denotes the Euclidean vector norm. The parameter $\theta$
appearing in (\ref{4.1}) is called the risk-sensitivity
parameter. The resulting estimate $\hat{x}_t$ obeys the
recursion (\ref{3.3})-(\ref{3.4}), where 
\beq
K_t = A(P_t^{-1} -\theta D^TD)^{-1} C^T (R_t^{\nu})^{-1}
\label{4.2}
\eeq
with
\beq
R_t^{\nu} = C(P_t^{-1}-\theta D^TD)^{-1}C^T + I
\:, \label{4.3}
\eeq
and where $P_t$ obeys the risk-sensitive Riccati equation
\beq
P_{t+1} = r^{\theta} (P_t) = A[P_t^{-1} +C^TC-\theta D^TD]^{-1}A^T
+ BB^T \: . \label{4.4}
\eeq
Our analysis will use the fact that the risk-sensitive
Riccati equation can be rewritten as
\beq
P_{t+1} = (A-K_tC)[P_t^{-1}-\theta D^TD]^{-1}(A-K_t C)^T 
+ BB^T +K_t K_t^T \: . \label{4.5}
\eeq

The values $\theta =0$, $\theta <0$ and $\theta >0$ of the 
risk-sensitivity parameter correspond respectively to the
risk-neutral, risk-seeking, and risk-averse cases. When
$\theta =0$, the risk-sensitive filter reduces to the 
Kalman fillter studied in the previous section, and when
$\theta <0$ the matrix $C^TC -\theta D^T D$ is non-negative
definite and can be rewritten as $\tilde{C}^T \tilde{C}$
where the pair formed by
\[
\tilde{C} \defn \bmat {c}
C \\
(-\theta)^{1/2} D 
\emat
\]
and $A$ is necessarily observable if $(C,A)$ is observable.
Accordingly, the convergence result of Theorem \ref{theo1}
is applicable to this problem, and in the remainder of this
paper our attention will focus on the risk-averse case
with $\theta >0$.

An interesting feature of the risk-sensitive filter is that
it can be interpreted as solving a standard least-squares
filtering problem in Krein space \cite{HSK2,HSK3}. We will
use this viewpoint here to extend the block filtering idea
of the previous section to the risk-sensitive case. The
Krein-space state-space model consists of dynamics (\ref{3.1})
and observations (\ref{3.2}), to which we must adjoin the
risk-sensitive observations
\beq
0 = D x_t + v_t^R \: . \label{4.6}
\eeq
The components of noise vectors $u_t$, $v_t$ and $v_t^R$ 
now belong to a Krein space and have the inner product
\beq
\left \langle \bmat {c}
u_t \\
v_t \\
v_t^R 
\emat
\, , \,
\bmat {c}
u_s \\
v_s \\
v_s^R 
\emat \right \rangle
= \bmat {ccc}
I_m & 0 & 0\\
0 & I_p & 0 \\
0 & 0 & -\theta^{-1} I_q \\
\emat \delta_{t-s} \label{4.7} 
\eeq
The $N$-step observability matrix of the pair $(D,A)$ 
is denoted as 
\[
{\cal O}_N^R = \bmat {cccc} 
(DA^{N-1})^T & \ldots & (DA)^T & D^T 
\emat^T 
\]
and if 
\[
L_t = \left \{ \begin{array} {cc}
DA^{t-1}B & t \geq 1 \\
0 & \mbox{otherwise}
\end{array} \right.
\]
denotes the impulse response from input $u_t$ to the 
risk-sensitive observation output, the corresponding
$N$-block Toeplitz matrix takes the form 
\[
{\cal L}_N = \bmat {cccccc}
0 & L_1 & L_2 & \cdots & L_{N-2} &L_{N-1}\\ 
0 & 0   & L_1 &  L_2 & \cdots & L_{N-2}\\ 
0 & 0   &  0 &  L_1 &  \cdots & L_{N-3}\\ 
\vdots & \vdots &  \vdots &  & \vdots & \vdots\\
0 & 0   & 0 & \cdots & \cdots & L_1 \\
0 & 0   & 0 & \cdots & \cdots & 0 
\emat \: .
\]
Then if
\[
{\bf v}_k^{R N} = \bmat {cccc}
(v_{kN+N-1}^R)^T & (v_{kN+N-2}^R)^T & \ldots & (v_{kN}^R)^T
\emat^T \: , 
\]
the $N$-block risk-sensitive observation for the downsampled
process $x_k^d$ can be expressed as
\beq
{\bf 0} = {\cal O}_N^R x_k^d + {\bf v}_k^{R N} + {\cal L}_N {\bf u}_k^N
\: . \label{4.8} 
\eeq
The Krein space inner product of observation noise vector
\[
\bmat{c}
{\bf w}_k^N\\
{\bf w}_k^{R N} 
\emat = \bmat {c} 
{\bf v}_k^N\\
{\bf v}_k^{R N} 
\emat + \bmat {c}
{\cal H}_N \\
{\cal L}_N
\emat {\bf u}_k^N
\]
with itself admits the block LDU decomposition
\bea
\lefteqn{\left \langle \bmat {c}
{\bf w}_k^N \\
{\bf w}_k^{R N}
\emat \, , \, \bmat {c}
{\bf w}_k^N \\
{\bf w}_k^{R N} 
\emat \right \rangle 
\defn {\cal K}_N^{\theta}}\nnl
&=& \bmat {cc}
I_{Np} & 0 \\
0 & -\theta^{-1} I_{Nq}
\emat + \bmat {c}
{\cal H}_N \\
{\cal L}_N \emat \bmat {cc}
{\cal H}_N^T & {\cal L}_N^T 
\emat \nnl
&=& \bmat {cc}
I_{Np} & 0 \\
{\cal L}_N{\cal H}_N^T (I_{Np} + {\cal H}_N{\cal H}_N^T)^{-1} & I_{Nq}
\emat 
\bmat {cc}
I_{Np} + {\cal H}_N{\cal H}_N^T & 0 \\
0 & S_N^{\theta}
\emat \nnl
&& \hspace*{1in} \times \bmat {cc}
I_{Np} & (I + {\cal H}_N{\cal H}_N^T)^{-1} {\cal H}_N{\cal L}_N^T \\
0 & I_{Nq}
\emat \: , \label{4.9}
\eea
where
\beq
S_N^{\theta} \defn -\theta^{-1} I_{Nq} + {\cal L}_N(I_{Nm}+{\cal H}_N^T{\cal
H}_N)^{-1}{\cal L}_N^T \label{4.10} 
\eeq
denotes the Schur complement of the $(1,1)$ block inside ${\cal
K}_N^{\theta}$. The projection of noise vector ${\bf u}_k^N$ on 
the Krein subspace spanned by the observation noise vector 
$\bmat {cc} ({\bf w}_k^N)^T & ({\bf w}_k^{R N})^T \emat^T$ is 
then given by 
\[
\hat{\bf u}_k^N = \bmat {cc}
{\cal G}_N^{\theta} & {\cal G}_N^{R \theta} 
\emat \bmat {c}
{\bf w}_k^N \\
{\bf w}_k^{N R}
\emat \: ,
\]
where 
\[
\bmat {cc}
{\cal G}_N^{\theta} & {\cal G}_N^{R \theta}
\emat = \bmat {cc}
{\cal H}_N^T & {\cal L}_N^T 
\emat ({\cal K}_N^{\theta})^{-1} \: ,
\]
and the residual $\tilde{\bf u}_k^N = {\bf u}_k^N -\hat{\bf u}_k^N$
has for inner product
\bea
\langle \tilde{\bf u}_k^N \, , \, \tilde{\bf u}_k^N \rangle
\defn {\cal Q}_N^{\theta}
&=& I_{Nm} - \bmat {cc}
{\cal G}_N^{\theta} & {\cal G}_N^{R \theta} 
\emat {\cal K}_N^{\theta}
\bmat {c}
({\cal G}_N^{\theta})^T \\ ({\cal G}_N^{R \theta})^T \emat  \nnl[2ex]
&=& [ I_{Nm} + {\cal H}_N^T {\cal H}_N - \theta {\cal L}_N^T {\cal
L}_N ]^{-1} \: . \label{4.11}
\eea
The matrix ${\cal Q}_N^{\theta} $ will be positive definite 
if and only if
\beq
\theta < \theta_N \defn 1/\lambda_1 ( {\cal L}_N (I_{Nm} + {\cal H}_N^T{\cal H}_N)^{-1} 
{\cal L}_N^T ) \: . \label{4.12}
\eeq
Note that this condition is also necessary and sufficient to ensure
that the Schur complement $S_N^{\theta}$ in (\ref{4.10}) is negative
definite. Then by multiplying the observation equation obtained by
combining equations (\ref{3.10}) and (\ref{4.8}) by 
${\cal R}_N \bmat {cc} {\cal G}_N^{\theta} & 
{\cal G}_N^{R \theta} \emat$ and subtracting it from
(\ref{3.9}), we obtain the state-space equation 
\beq
x_{k+1}^d = \alpha_N^{\theta} x_k^d + {\cal R}_N \tilde{\bf u}_k^N
+{\cal R}_N {\cal G}_N^{\theta} {\bf y}_k^N
\label{4.13}
\eeq
with
\[
\alpha_N^{\theta} \defn A^N - {\cal R}_N [ {\cal G}_N^{\theta} {\cal O}_N 
+ {\cal G}_N^{R \theta} {\cal O}_N^R ] \: , 
\]
where the driving noise is now orthogonal to the noises ${\bf w}_k^N$
and ${\bf w}_k^{R N}$ appearing in observation equations (\ref{3.10})
and (\ref{4.8}). Accordingly, the Riccati equation associated to the
downsampled model takes the form 
\beq
P_{k+1}^d = r_d^{\theta} (P_k^d) \defn \alpha_N^{\theta}
[ (P_k^d)^{-1} + \Omega_N^{\theta} ]^{-1} ( \alpha_N^{\theta})^T
+ W_N^{\theta} \: , \label{4.14}
\eeq
where 
\bea
\Omega_N^{\theta} &=& \bmat {cc} 
{\cal O}_N^T & ({\cal O}_N^R)^T 
\emat ({\cal K}_N^{\theta})^{-1} \bmat {c}
{\cal O}_N \\
{\cal O}_N^R
\emat \nnl 
&=& \Omega_N + {\cal J}_N^T (S_N^{\theta})^{-1} {\cal J}_N
\label{4.15}
\eea
with
\[
{\cal J}_N \defn {\cal O}_N^R  -{\cal L}_N{\cal H}_N^T [I+{\cal H}_N{\cal
H}_N^T]^{-1}{\cal O}_N 
\]
and
\beq
W_N^{\theta} = {\cal R}_N{\cal Q}_N^{\theta}{\cal R}_N^T 
\:. \label{4.16}
\eeq

For $\theta=0$, the matrices $\Omega_N^{\theta}$ and 
$W_N^{\theta}$ coincide with the risk-neutral Gramians 
$\Omega_N$ and $W_N$ defined in (\ref{3.13}) and (\ref{3.14}).
These matrices are positive definite for $N \geq n$ 
if and only if the pairs $(C,A)$ and $(A,B)$ are observable
and reachable, respectively. Since ${\cal Q}_N^{\theta}$
is positive definite for $0 \leq \theta < \theta_N$, we
deduce that $W_N^{\theta} >0$ over this range as long
as $(A,B)$ is reachable and $N \geq n$. On the other 
hand, the Schur complement matrix $S_N^{\theta}$
is negative definite for $0 \leq \theta < \theta_N$, so
\[
\Omega_N^{\theta} < \Omega_N \label{4.17}
\]
over this range. To establish that there exists a
range $0 \leq \theta < \tau_N$ over which $\Omega_N^{\theta}$
remains positive definite when $(C,A)$ is observable,
we use the following observation.  
\vskip 2ex

\begin{lemma} 
\label{lem2}
Over $0 \leq \theta < \theta_N$, the 
Gramians $\Omega_N^{\theta}$ and $W_N^{\theta}$ are 
monotone decreasing, and monotone nondecreasing, respectively,
with respect to the partial order defined on nonnegative 
definite matrices.
\end{lemma}
\vskip2ex

\begin{proof} We have
\[
\frac{d~}{d\theta} (S_N^{\theta})^{-1}= -(S_N^{\theta})^{-1} 
\Big(\frac{d~}{d\theta}S_N^{\theta} \Big) (S_N^{\theta})^{-1} 
= -(\theta S_N^{\theta})^{-2} < 0 
\]
and 
\[
\frac{d~}{d\theta} Q_N^{\theta}= -Q_N^{\theta}
\frac{d~}{d\theta} (Q_N^{\theta})^{-1} Q_N^{\theta} 
= Q_N^{\theta}{\cal L}_N^T {\cal L}_N Q_N^{\theta} \geq 0 \: .  
\]
\qquad
\end{proof}
\vskip 2ex

To understand why $\Omega_N^{\theta}$ and $W_N^{\theta}$
vary in opposite direction as $\theta$ increases, note that
$W_N^{\theta}$ can be viewed as a measure of the uncertainty
introduced by the process noise in the state-space model,
whereas $\Omega_N^{\theta}$ is a measure of the information
about the state contained in a block observation. As the
risk-sensitivity parameter $\theta$ increases, it is natural 
that the uncertainty matrix $W_N^{\theta}$ should increase and 
the information matrix $\Omega_N^{\theta}$ should decrease.

Let $\tau_N < \theta_N$ be the first value of $\theta$ for 
which $\Omega_N^{\theta}$ becomes singular. Then since
$\Omega_N^{\theta}$ and $W_N^{\theta}$ are positive definite
for $\theta \in [0,\tau_N)$, we conclude that over this range
the Riccati map $r_d^{\theta}$ is strictly contractive and has
a unique fixed point $P$ in ${\cal P}$. Like the
risk-neutral case, we have $r_d^{\theta} = (r^{\theta})^N$.
However, because the image $r^{\theta}({\cal P})$ is not
completely contained in ${\cal P}$, to ensure that $P$
is also the unique fixed point of $r^{\theta}$, we must also
require that $r^{\theta} (P) \in {\cal P}$. Note indeed that
if 
\beq
P = (r^{\theta})^N (P) \: , \label{4.18}
\eeq
by applying $r^{\theta}$ to both sides of (\ref{4.18}), we 
obtain
\[
r^{\theta}(P) = (r^{\theta})^N (r^{\theta} (P))
\]
so $r^{\theta}(P)$ is a fixed point of $(r^{\theta})^N$. If
$r^{\theta}(P) \in {\cal P}$, we must have
\[
P= r^{\theta} (P)
\]
since $(r^{\theta})^N$ has a unique fixed point in ${\cal P}$.  

At this point it is worth pointing out that until now we
have ignored an important constraint \cite{BS,HSK1} for the 
risk-sensitive filter, namely that the matrix
\beq
V_t = (P_t^{-1} - \theta D^TD)^{-1} \label{4.19}
\eeq
should be positive definite for all $t$. If this condition 
is satisfied, then the fixed point $P$ of $r_d^{\theta}$ will
be in ${\cal P}$, ensuring that it is the unique fixed
point of $r^{\theta}$.

\section{Positiveness conditions for $V_t$}
\label{sec:posit}

In this section we identify conditions on the initial
covariance $P_0$ and risk-sensitivity parameter $\theta$ which 
ensure that the trajectory of iteration $P_{t+1} = r^{\theta}
(P_t)$ satisfies $V_t >0$ for all $t$. Our analysis will
exploit the monotonicity of Riccati operator $r^{\theta} (P)$
with respect to the partial order of positive definite matrices. 
\vskip 2ex

\begin{lemma} 
\label{lem3}
Let $P_1$ and $P_2$ be two matrices
in ${\cal P}$ such that $P_1 \geq P_2$ and $P_1^{-1}
- \theta D^TD > 0$. Then
\beq
r^{\theta} (P_1) \geq r^{\theta} (P_2) \: . \label{5.1}
\eeq
\end{lemma}

\begin{proof}
The monotonicity of $r^{\theta}$ is due to the fact
that the inversion of positive definite matrices 
reverses their partial order. In addition, congruence 
transformations and translation by symmetric matrices 
preserve the partial order. Since the operator   
$r^{\theta}(P)$ in (\ref{4.4}) can be expressed in terms of
two nested inversions of positive definite matrices, two
matrix translations and a congruence transformation, it
is monotone in $P$.\qquad
\end{proof}
\vskip 2ex

Next, observe that for any $n \times p$ observer gain
matrix $G$, the risk-sensitive Riccati equation (\ref{4.5})
can be rewritten as 
\bea
P_{t+1} &=& (A -GC)(P_{t}^{-1}- \theta D^TD)^{-1} (A-GC)^T
+ GG^T +BB^T \nnl
&& - [(A-GC)(P_t^{-1}-\theta D^TD)^{-1}C^T -G](R_t^{\nu})^{-1} \nnl
&& \hspace*{0.5in} \times [(A-GC)(P_t^{-1} - \theta D^T D)^{-1}C^T -G]^T \: .
\label{5.2} 
\eea
This expression can be obtained by writing $A= (A-GC)+GC$
in (\ref{4.5}) and performing simple algebraic manipulations.
While it may appear surprising that a free matrix gain $G$
can be introduced in the equation, the above modification has
actually a simple explanation. Consider the state-space model
(\ref{3.1})--(\ref{3.2}). We can always design a preliminary 
suboptimal observer
\beq
\hat{x}_{t+1}^S = A \hat{x}_t^S + G(y_t -C\hat{x}_t^S)
\: . \label{5.3}
\eeq
Then the residual $\tilde{x}_t^S = x_t - \hat{x}_t^S$ admits
the state-space model
\bea
\tilde{x}_{t+1}^S &=& (A-GC) \tilde{x}_t^S + Bu_t -Gv_t \nnl
y_t - C \hat{x}_t^S &=&  C \tilde{x}_t^S + v_t \: , \label{5.4}
\eea
for which the only difference with respect to the original model
(\ref{3.1})--(\ref{3.2}) is that the process noise $Bu_t -Gv_t$ and
measurement noise $u_t$ are now correlated. The risk-neutral
and risk-sensitive problems associated to the original model
(\ref{3.1})--(\ref{3.2}) and modified model (\ref{5.4}) are 
exactly the same since observations $y_t$ and 
\beq
y_t^S \defn y_t - C\hat{x}_t^S \label{5.5} 
\eeq
can be obtained causally from each other. In particular, the
variance matrices $P_t$ of the error are the same for both
models. Thus it should not be a surprise that the solution
$P_t$ of Riccati equation (\ref{4.5}) should also solve
the risk-sensitive Riccati equation (\ref{5.2}) corresponding 
to modified model (\ref{5.4}).  

One important advantage of introducing the free matrix gain $G$
is that when the pair $(C,A)$ is observable, the characteristic
polynomial of the closed-loop observer matrix $A-GC$ can be 
assigned arbitrarily \cite{Kai}. In particular, it is possible
to ensure that the matrix $A-GC$ is stable, i.e. all its 
eigenvalues are strictly inside the unit circle. In this
case, let
\[
r \defn \max_{1 \leq i \leq n} |\lambda_i (A-GC)|
\]
denote its spectral radius. For $\rho <1/r$, the 
matrix $\rho (A-GC)$ will also be stable, and when $(A,B)$
is reachable, the algebraic Lyapunov equation (ALE)
\beq
\Sigma_{\rho} = \rho^2 (A-GC) \Sigma_{\rho} (A-GC)^T 
+ BB^T + GG^T \label{5.6}
\eeq
admits a unique positive definite solution 
\beq
\Sigma_{\rho} = \sum_{k=0}^{\infty} \rho^{2k} (A-GC)^k (BB^T + GG^T)
((A-GC)^k)^T \: . \label{5.7}
\eeq  
Note that $\Sigma_{\rho}$ is positive definite if and only if 
the pair $A-GC$, $\bmat {cc} B & G \emat$ is reachable. But if
this pair is not reachable, by the Popov-Belevich-Hautus
(PBH) test \cite[p. 366]{Kai}, there must be a left eigenvector
$z^T$ of $A-GC$ which is orthogonal to the column space of
$\bmat {cc} B & G \emat$, so 
\[
z^T (A-GC) = \lambda z^T \hspace*{0.15in}, \hspace*{0.15in}
z^T B = z^T G = 0 \: .
\]
This implies $z^T A = \lambda z^T$, so $z^T$ is a left eigenvector
of $A$ perpendicular to the column space of $B$, which implies
that $(A,B)$ is not reachable, a contradiction.

If we select $1 < \rho < 1/r$, the matrix $\Sigma_{\rho}$ 
is positive definite and the matrix
\beq
M \defn (1-\rho^{-2}) \Sigma_{\rho}^{-1} - \theta D^T D 
\label{5.8}
\eeq
will be non-negative definite if and only if the matrix
\[
\tilde{M} \defn I_n -\theta \frac{\rho^2}{\rho^2 -1} \Sigma_{\rho}^{1/2} D^TD
\Sigma_{\rho}^{1/2} 
\]
is non-negative definite. But because the matrices $\Sigma_{\rho}^{1/2}
D^T D \Sigma_{\rho}^{1/2}$ and $D\Sigma_{\rho}D^T$ have the same
nonzero eigenvalues, $\tilde{M}$ is non-negative definite if and only if
\beq 
\theta \frac{\rho^2}{\rho^2 -1} D \Sigma_{\rho} D^T \leq  I_q
\label{5.9}
\eeq
or equivalently
\beq
0 \leq \theta \leq  \beta_{\rho} \defn \frac{\rho^2-1}{ \rho^2 \lambda_1 
(D \Sigma_{\rho} D^T)} \label{5.10}
\eeq
where $\lambda_1 (D \Sigma_{\rho} D^T)$ is the largest 
eigenvalue of $D\Sigma_{\rho} D^T$. It is strictly positive 
since $\Sigma_{\rho}$ is positive definite and $D$ has full 
row rank.
\vskip 2ex

\begin{lemma}
\label{lem4}
If the initial variance $P_0$ for the risk-sensitive Riccati
equation (\ref{4.4}) (or equivalently (\ref{5.2})) satisfies
$0 < P_0 \leq \Sigma_{\rho}$ and $0 \leq \theta \leq \beta_{\rho}$,
the entire trajectory of the recursion $P_{t+1} = r^{\theta}
(P_t)$ satisfies $0 < P_t \leq \Sigma_{\rho}$, so $V_t >0$.
Furthermore, for $P_0 = \Sigma_{\rho}$, the sequence $P_t$
is monotone decreasing.
\end{lemma}
\vskip 2ex

\begin{proof}
Suppose first that $P_0 = \Sigma_{\rho}$. Then the
non-negative definiteness of the matrix $M$ in (\ref{5.8})
implies $V_0 >0$. By subtracting (\ref{5.2}) for $t=0$ from
ALE (\ref{5.6}), we obtain
\bea
\lefteqn{\Sigma_{\rho} - P_1 =
(A-GC)(\rho^2 \Sigma_{\rho}
- (\Sigma_{\rho}^{-1} - \theta D^TD)^{-1})(A-GC)^T} \nnl
&&+ [(A-GC)(\Sigma_{\rho}^{-1}-\theta D^TD)^{-1}C^T -G]\nnl
&& \times (R_0^{\nu})^{-1}
[(A-GC)(\Sigma_{\rho}^{-1} - \theta D^T D)^{-1}C^T -G]^T \: .
\label{5.11} 
\eea
But when $M$ is non-negative definite, the matrix
\[
\rho^2 \Sigma_{\rho} - (\Sigma_{\rho}^{-1} -\theta D^T D)^{-1}
\]
appearing in the first term of the right hand side of (\ref{5.11})
is non-negative definite, which implies 
\beq
P_1 = r^{\theta} (\Sigma_{\rho}) \leq P_0= \Sigma_{\rho} \:.
\label{5.12}
\eeq 
By induction, suppose that $P_t \leq P_{t-1}$. The motonicity
of $r^{\theta}$ implies
\[
P_{t+1} = r^{\theta} (P_t) \leq r^{\theta} (P_{t-1}) =P_t
\]
so $P_t$ is monotone decreasing.

Next, consider the case of an initial condition $P_0 \leq
\Sigma_{\rho}$. The monotonicity of $r^{\theta}$ implies
\[
P_1  = r^{\theta} (P_0) \leq r^{\theta} (\Sigma_{\rho}) 
\leq \Sigma_{\rho}
\]
where the last inequality uses (\ref{5.12}). Proceeding
by induction, we deduce that $P_t \leq \Sigma_{\rho}$
for all $t$. This implies
\[
P_t^{-1} \geq \Sigma_{\rho}^{-1} > \theta D^TD
\]
so $V_t > 0$ for all $t$.\qquad
\end{proof}
\vskip 2ex

\noindent
{\it Remarks:} 

\begin{remunerate}

\item[1)] For the risk-neutral case ($\theta=0$),
the solution $\Sigma_{\rho}$ of the ALE (\ref{5.6}) is
similar to an upper bound proposed for the positive 
definite solution of the algebraic Riccati equation (ARE) 
in \cite{DSW} (see also \cite{KP}), which was also shown
to yield a monotone decreasing sequence of iterates. 
However the construction of the upper bound given in
\cite{DSW} is purely algebraic, whereas for $\rho=1$
the covariance matrix $\Sigma_{\rho}$ can be interpreted
as the steady-state error variance of the suboptimal
filter (\ref{5.3}).

\item[2)] Since the bound $\beta_{\rho}$ for the
risk-sensitivity parameter depends on both $G$ and $\rho$,
it is of interest to determine if a choice of $G$
and $\rho$ makes the bound as large as possible. Note in
this respect that there exists a trade-off between
making $\beta_{\rho}$ as large as possible and enlarging    
the set $0 \leq  P_0 \leq \Sigma_{\rho}$ of allowable
initial conditions, since from (\ref{5.9}) in order to 
increase the range of $\theta$ values, $\Sigma_{\rho}$
must be as small as possible, which shrinks the domain of
allowable $P_0$s. A clue on how to select $G$ is provided
by the scalar case analysis presented in \cite[Chap. 9]{Whi}.
With $n=m=p=q=1$, if we select the gain $G=A/C$,
$A-GC=0$ so $\rho$ can be selected arbitrarily large,
and 
\[
\Sigma_{\rho} = \frac{A^2}{C^2} + B^2  
\]
for all $\rho$. Letting $\rho \rightarrow \infty$ in
(\ref{5.10}), the bound $\beta_{\rho}$ then coincides 
with the scalar case bound derived on p. 116 of
\cite{Whi}. This suggests that selecting a gain $G$
that moves all the eigenvalues of the closed-loop 
observer $A-GC$ to zero is likely to yield a
satisfactory upper bound $\beta_{\rho}$. Note that in the
multivariable case, $A-GC$ cannot in general be
set to zero by selecting the gain matrix $G$, but
the characteristic polynomial and some additional
parameters (when $p>1$) can be assigned arbitrarily
\cite[Chap. 7]{Kai}. Unfortunately, as will be
demonstrated on an example in the next section,
the gain $G$ which assigns all the eigenvalues
of $A-GC$ to zero does not necessarily yield the largest 
possible value of $\beta_{\rho}$ and a comprehensive
search over $G$ and $\rho$ is usually required to
make $\beta_{\rho}$ as large as possible.  

\end{remunerate}

By assembling the preliminary results of the current
and previous sections, we obtain the following convergence
theorem for risk-sensitive filters.
\vskip 2ex

\begin{theorem}
\label{theo2}
Assume that in system (\ref{3.1})--(\ref{3.2}), the pairs 
$(A,B)$ and $(C,A)$ are reachable and observable. Then
if $0 \leq \theta < \tau_N$ and $\theta \leq \beta_{\rho}$
with $N \geq n$, the risk-sensitive Riccati map $r^{\theta}$
has a unique positive definite fixed point $P$ such
that $P^{-1} - \theta D^TD >0$. Furthermore, if the
initial condition $P_0$ of the Riccati equation 
satisfies $0 < P_0 \leq \Sigma_{\rho}$, the entire 
trajectory of iteration $P_{t+1} = r^{\theta} (P_t)$
stays in ${\cal P}$, satisfies $V_t >0$ and tends to $P$.
In this case the limit $K$ of filtering gain
$K_t$ as $t \rightarrow \infty$ has the property that 
$A-KC$ is stable.
\end{theorem}

\begin{proof}
Since the trajectory $P_t$ stays in
${\cal P}$ and satisfies $V_t > 0$, and the $N$-fold
operator $r_d^{\theta} = (r^{\theta})^N$ has a unique
fixed point $P$ in ${\cal P}$, the sequence $P_t$
must tend to $P$, and $P$ must be such that $P^{-1}
- \theta D^TD >0$. Then the stability of $A-KC$ can be
established by applying Lyapunov stability theory to 
the risk-sensitive ARE.
\[
P = (A-KC)(P^{-1} -\theta D^TD)^{-1}(A-KC)^T + BB^T + KK^T \: .
\]
\qquad
\end{proof} 
\vskip 2ex

This theorem answers in the affirmative the question 
posed in \cite{BS} whether it is possible to specify
a-priori a range of risk-sensitivity parameters
$\theta$ and initial conditions such that the 
risk-sensitive Riccati equation admits a solution.
On the other hand, it leaves open the computation of
the maximum value of $\theta$ (its breakdown value in
the terminology of \cite{Whi}) for which a solution
exists, which corresponds to the optimal $H^{\infty}$
filter.

\section{Example}
\label{sec:examp}

To illustrate our results, we consider a system with
\[
A= \bmat {cc}
0.1 & 1\\
0 & 1.2
\emat 
\hspace*{0.3in} C = \bmat {cc} 1 & -1 \emat
\]
and $B=D=I_2$. Note that $A$ is unstable, $(A,B)$ is reachable,
but the pair $(C,A)$ is barely observable, since the eigenvector
\[
p= \bmat {c}
1 \\ 1.1
\emat
\] 
corresponding to the eigenvalue $\lambda=1.2$ is at
a $92.72$ degree angle with respect to $C$. In this case,
for $N=2$, the largest eigenvalue of matrix ${\cal L}_2
(I_4 + {\cal H}_2^T{\cal H}_2)^{-1}{\cal L}_2$ equals $1$, 
so $\theta_2 =2$. To evaluate the value $\tau_2$ for which
$\Omega_2^{\theta}$ becomes singular, the smallest eigenvalue
of Gramian $\Omega_2^{\theta}$ is plotted in \fig{laOm2}
as a function of $\theta$ for $0 \leq \theta \leq 2 \times 10^{-3}$.
For this example, it decreases linearly, and becomes
negative at $\tau_2 = 0.715 \times 10^{-3}$. For completeness,
the smallest eigenvalue of reachability Gramian $W_2^{\theta}$
is plotted in \fig{laW2} over the same range of $\theta$.
It is monotone increasing, as expected, but the rate of increase
is very small, since $\lambda_2(W_2^{\theta})$ varies from
$1.002828$ to $1.02831$. Note that although we have selected
$N=2$ here, larger values of $N$ can be considered, and in
fact as $N$ increases, $\tau_N$ increases and $\theta_N$
decreases, and for this example both values tend to
$1.33 \times 10^{-3}$ for large $N$.
\vskip 2ex

\bef[htb]
\centering
\includegraphics[width =3.5in, height=3in]{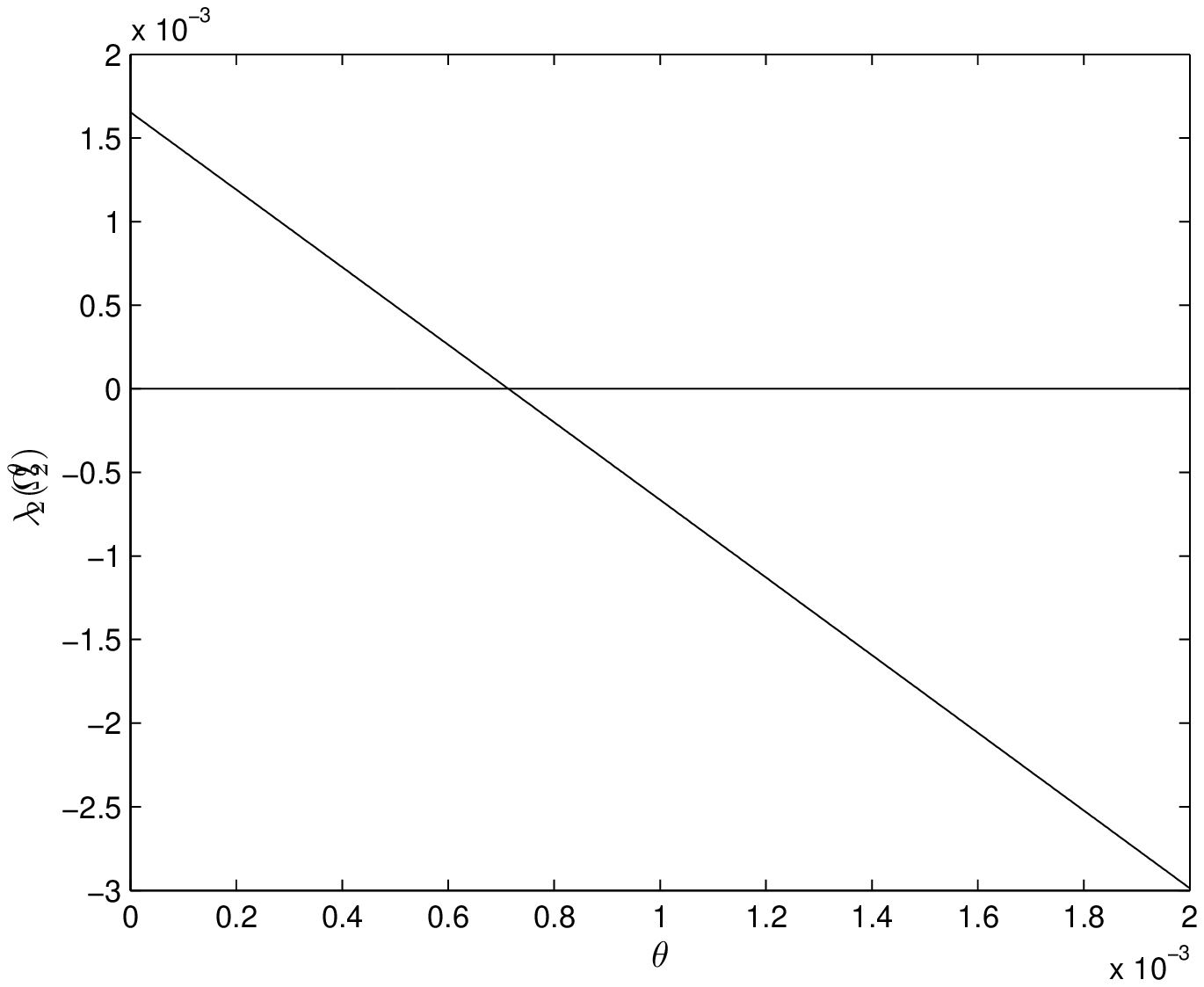}
\caption{Smallest eigenvalue of observability Gramian
$\Omega_2^{\theta}$ for $0 \leq \theta \leq 2 \times 10^{-3}$.}
\label{laOm2}
\eef

\bef[htb]
\centering
\includegraphics[width =3.5in, height=3.5in]{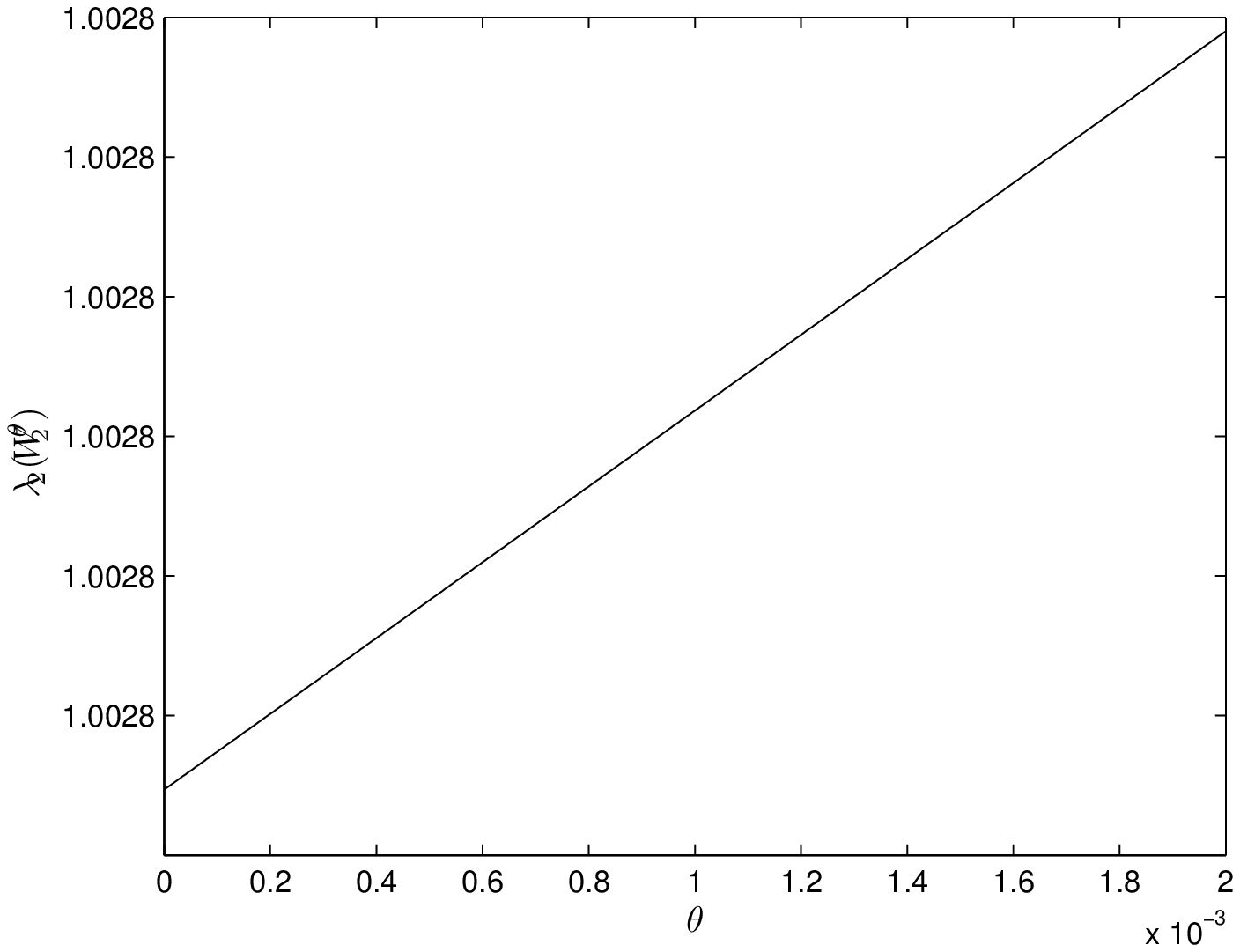}
\caption{Smallest eigenvalue of reachability Gramian
$W_2^{\theta}$ for $0 \leq \theta \leq 2 \times 10^{-3}$.}
\label{laW2}
\eef

Next, to evaluate $\beta_{\rho}$, we observe that with
the gain matrix
\beq
G= \bmat {c}
-13.1\\
-14.4
\emat
\label{6.1}
\eeq
the closed-lood matrix 
\[
A-GC = \bmat {cc}
13.2 & -12.1\\
14.4 & -13.2
\emat
\]
is nilpotent, i.e., its eigenvalues are zero. Note however that
$G$ is rather large, which reflects the weak observability 
of the system. In this case, if we select $\rho=2$, the
solution $\Sigma_2$ of the Lyapunov equation (\ref{5.6}) is
\[
\Sigma_2 = 10^{3} \bmat {cc} 
1.4622 & 1.5954 \\
1.5954 &  1.7431
\emat \: .
\]  
Its largest eigenvalue is $\lambda_1 (\Sigma_2) = 3.2042 \times 10^3$
and from (\ref{5.10}), we obtain $\beta_2 = 2.3407 \times 10^{-4}$.
This bound is significantly smaller than $\tau_2$. To
illustrate Lemma \ref{lem4}, the risk-sensitive Riccati iteration
$P_{t+1} = r^{\theta} (P_t)$ is simulated with 
$\theta = \beta_2$ and initial condition $P_0 = \Sigma_2$.
The two eigenvalues of $P_t$ and $V_t$ are plotted as
a function of $t$ for $0 \leq t \leq 10$ in \fig{eigP}
and \fig{eigV}, respectively. As expected, the eigenvalues 
remain positive and are monotone decreasing. The 
monotone decreasing property of the eigenvalues is due
to the fact that if two $n \times n$ positive definite 
matrices $P$ and $Q$ are such that $P \geq Q$ and if the 
eigenvalues of $P$ and $Q$ are sorted in decreasing
order, then $\lambda_i (P) \geq \lambda_i (Q)$ for $1 \leq i
\leq n$. In other words, the eigenvalues follow the
partial order of positive definite matrices. Since 
according to Lemma \ref{lem4}, the sequence $P_t$ is monotone
decreasing, so are its eigenvalues. The figures indicate
that the risk-sensitive Riccati equation converges very 
quickly, after 4 or 5 iterations. Note that 
if $P$ denotes the limit of $P_t$, its smallest
eigenvalue is $1.003$, but the other eigenvalue is
much larger and equals $332.4$. This reflects our earlier 
observation that one of the modes of the system is barely 
observable. The eigenvalues of the matrix $A-KC$ for
the estimation error dynamics are $0.034$ and $0.776$,
so the filter is stable, as expected.
\vskip 2ex

\bef[htb]
\centering
\includegraphics[width =4in, height=6in]{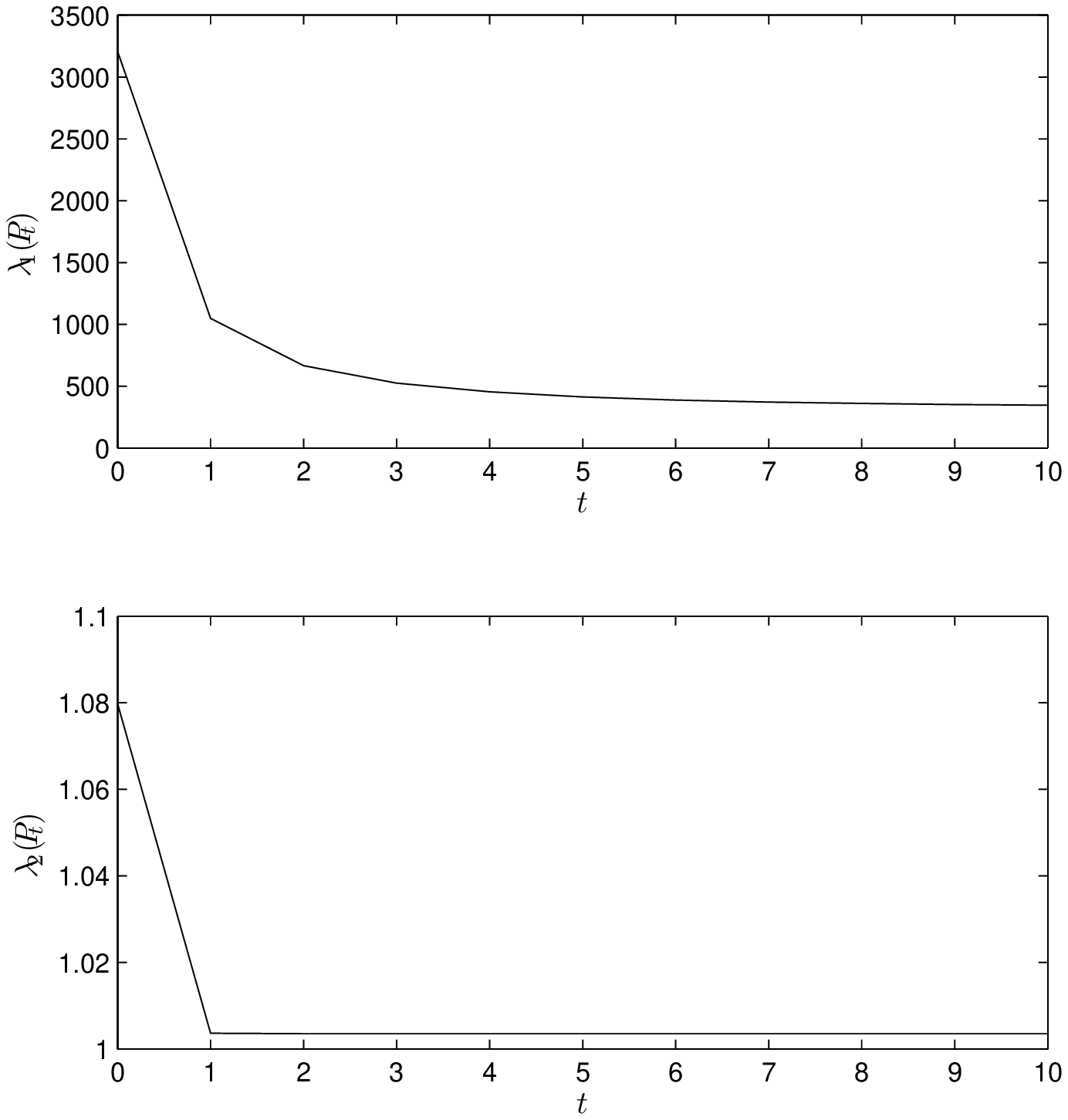}
\caption{Eigenvalues of Riccati solution $P_t$ for 
$0 \leq t \leq 11$ with $\theta = \beta_2$ and initial 
condition $P_0 = \Sigma_2$.}
\label{eigP}
\eef

\bef[htb]
\centering
\includegraphics[width =4in, height=6in]{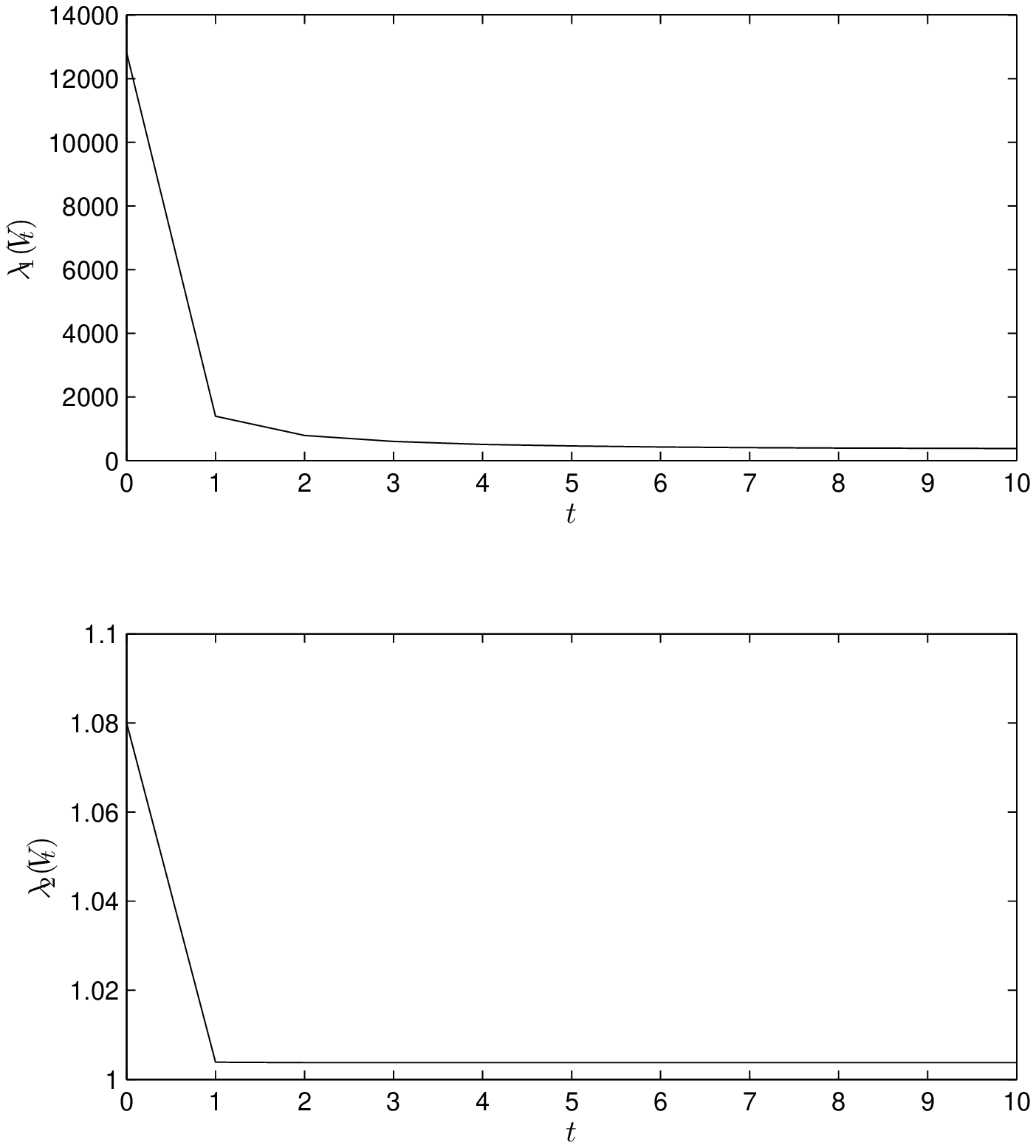}
\caption{Eigenvalues of $V_t$ for $0 \leq t \leq 11$ 
with $\theta = \beta_2$ and initial condition  
$P_0 = \Sigma_2$.}
\label{eigV}
\eef

\noindent

Finally, to illustrate the onset of breakdown as $\theta$
increases, the two eigenvalues of the fixed point solution
$P^{\theta}$ of $r^{\theta}$ and of the corresponding matrix
$V^{\theta} = ((P^{\theta})^{-1} - \theta I_2)^{-1}$ are plotted
as a function of $\theta$ in \fig{thdepP} and \fig{thdepV},
respectively, for $0 \leq \theta \leq 0.95 \times 10^{-3}$.
It is known \cite[p. 379]{HSK1} that $P^{\theta}$ is a
monotone increasing function of $\theta$, and as expected
the eigenvalues of $P^{\theta}$ are monotone increasing.
However, while the change in the smaller eigenvalue
is barely noticeable, the eigenvalue representing the
weakly observable mode increases rapidly with $\theta$.
As $\theta$ increases, the eigenvalues of $V^{\theta}$
start diverging, and the breakdown value of
$\theta$ for this example is just above $0.95
\times 10^{-3}$. This value is significantly higher
than the bound $\beta_2$ obtained by applying Lemma \ref{lem4}
with the gain (\ref{6.1}), suggesting that the bound can be 
improved. In fact, an exhaustive search over $G$ and $\rho$ 
showed that $\beta_{\rho}$ is maximized by selecting
\[
G = \bmat {c}
   -7.2196\\
  -7.9753
\emat
\]
and $\rho=1.2849$, in which case $\beta_{\rho} = 0.4824 \times
10^{-3}$.
\vskip 2ex

\bef[htb]
\centering
\includegraphics[width =4in, height=6in]{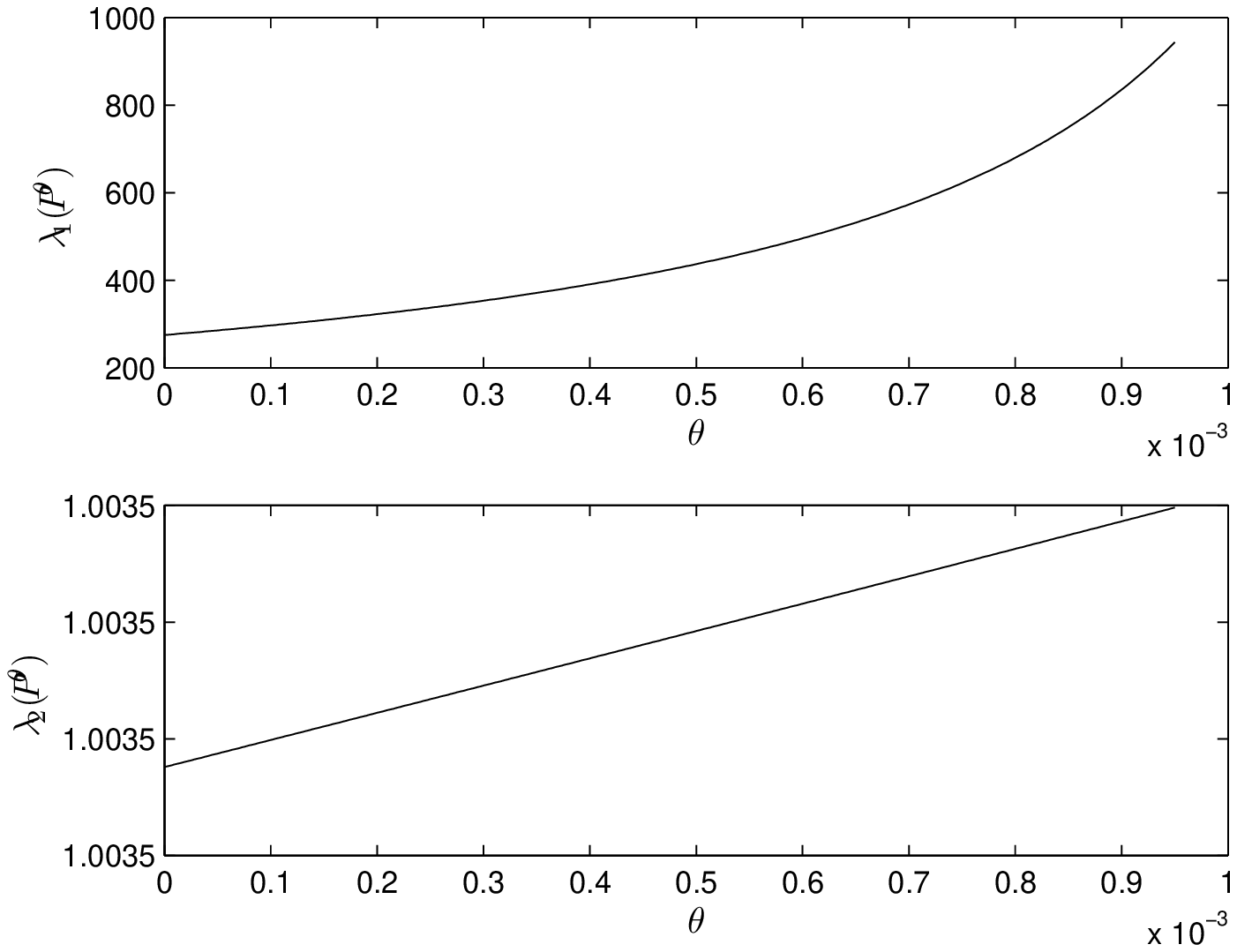}
\caption{Eigenvalues of Riccati fixed point $P^{\theta}$
in function of $\theta$ for $0 \leq \theta \leq 0.95 \times 10^{-3}$.}
\label{thdepP}
\eef

\bef[htb]
\centering
\includegraphics[width =4in, height=6in]{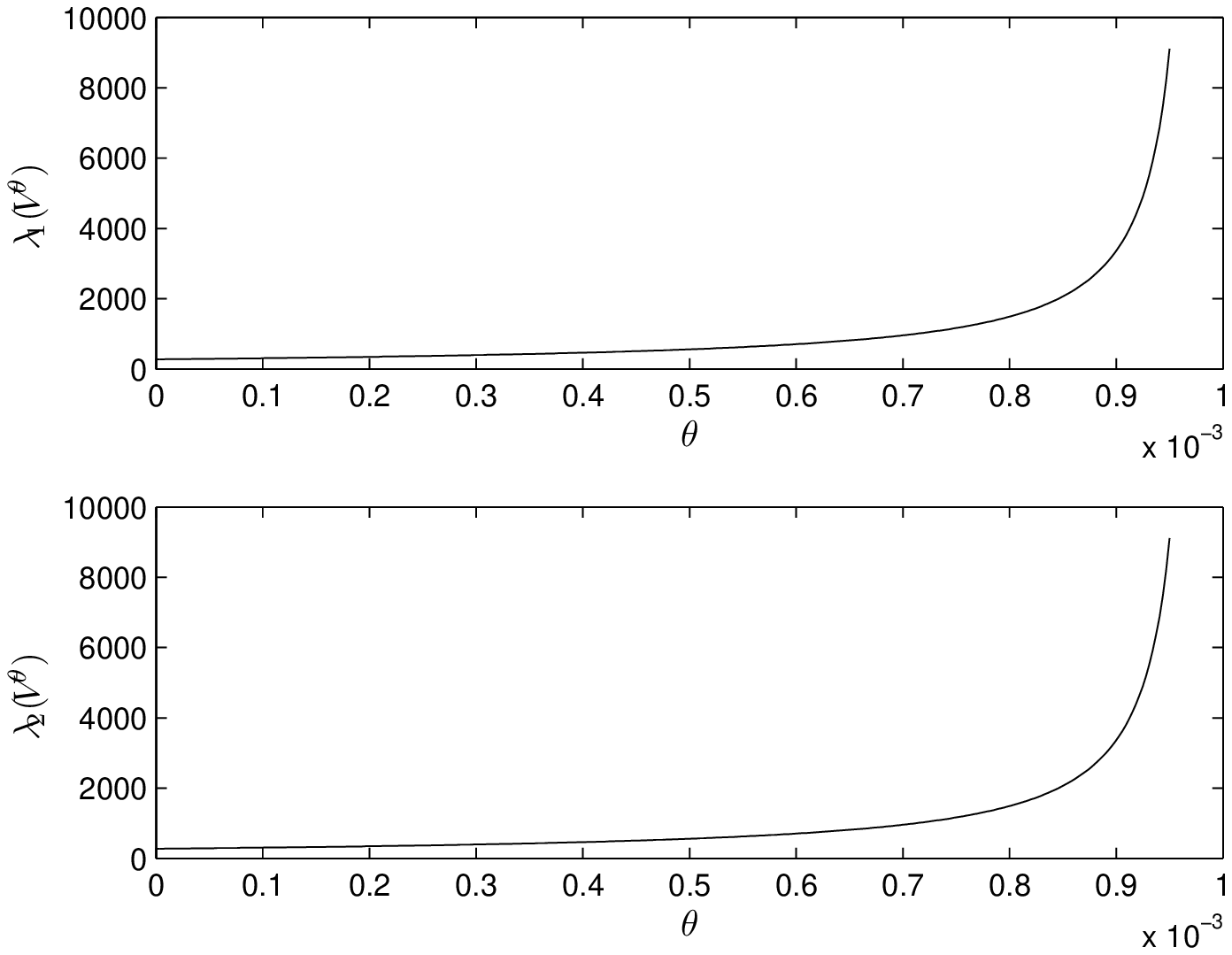}
\caption{Eigenvalues of $V^{\theta}$ in function of
$\theta$ for $0 \leq \theta \leq 0.95 \times 10^{-3}$.}
\label{thdepV}
\eef

\section{Conclusions}
\label{sec:conc}

A convergence analysis of risk-sensitive filters has been presented.
It relies on extending Bougerol's contraction analysis
of risk-neutral Riccati equations to the risk-sensitive case.
This was accomplished by considering a block-filtering
implementation of the $N$-fold Riccati map and showning that
this map is strictly contractive as long as an observability
Wronskian depending on the risk-sensitivity parameter 
remains positive definite. A second condition was derived for the 
risk-sensitivity parameter and initial error variance to
ensure that the trajectory of the risk-sensitive Riccati
iteration stays positive definite at all times. The
two conditions obtained can be viewed as multivariable
versions of conditions obtained earlier by Whittle 
\cite[Chap. 9]{Whi} for the scalar case.

Although the results we have presented concern filters
with a constant risk-sensitivity parameter $\theta$,
a closely related class of robust filters was derived
recently \cite{LN} by assigning a fixed relative entropy tolerance
to increments of the state-space model. In this case, the
risk-sensitivity parameter is time-varying, but the tolerance
is fixed, and based on computer simulations, it appears
that the risk-sensitivity parameter and associated filter
always converge as long as the relative entropy tolerance
remains small. Since Bougerol's analysis \cite{Bou} is 
applicable to systems with random fluctuations, it is 
reasonable to wonder if the analysis presented here
can be extended to establish the convergence of the
filters discussed in \cite{LN}.

\bibliography{contrac}

\begin{thebibliography}{10}

\bibitem{AM}
{\sc B.~D.~O. Anderson and J.~B. Moore}, {\em \it Optimal Filtering},
  Prentice-Hall, Englewood Cliffs, NJ, 1979.

\bibitem{AE}
{\sc J.~P. Aubin and I.~Ekeland}, {\em Applied Nonlinear Analysis}, J. Wiley,
  New York, 1984.

\bibitem{BS}
{\sc R.~N. Banavar and J.~L. Speyer}, {\em Properties of risk-sensitive
  filters/estimators}, IEE Proc.-Control Theory Appl., 145 (1998).

\bibitem{Bha}
{\sc R.~Bhatia}, {\em On the exponential metric increasing property}, {Linear
  Algebra and its Appl.}, 375 (2003), pp.~211--220.

\bibitem{Bou}
{\sc P.~Bougerol}, {\em Kalman filtering with random coefficients and
  contractions}, {SIAM J. Control and Optimiz.}, 31 (1993), pp.~942--959.

\bibitem{DSW}
{\sc R.~Davies, P.~Shi, and P.~Wiltshire}, {\em New upper solution bounds of
  the discrete algebraic {R}iccati matrix equation}, J. Computational and
  Applied Math., 213 (2008), pp.~307--317.

\bibitem{HS}
{\sc L.~P. Hansen and T.~J. Sargent}, {\em Robustness}, Princeton University
  Press, Princeton, NJ, 2008.

\bibitem{HSK2}
{\sc B.~Hassibi, A.~H. Sayed, and T.~Kailath}, {\em Linear estimation in
  {K}rein spaces. {I}. {T}heory}, {IEEE Trans. Automat. Control}, 41 (1996),
  pp.~18--33.

\bibitem{HSK3}
\leavevmode\vrule height 2pt depth -1.6pt width 23pt, {\em Linear estimation in
  {K}rein spaces. {II}. {A}pplications}, {IEEE Trans. Automat. Control}, 41
  (1996), pp.~34--49.

\bibitem{HSK1}
\leavevmode\vrule height 2pt depth -1.6pt width 23pt, {\em Indefinite-Quadratic
  Estimation and Control-- A Unified Approach to $H^2$ and $H^{\infty}$
  Theories}, Soc. Indust. Appl. Math., Philadelphia, 1999.

\bibitem{ISY}
{\sc M.~Ito, Y.~Seo, T.~Yamazaki, and M.~Yanagida}, {\em Geometric properties
  of positive definite matrices cone with respect to the {T}hompson metric},
  Linear Algebra and its Appl., 435 (2011).

\bibitem{Kai}
{\sc T.~Kailath}, {\em Linear Systems}, Prentice Hall, Englewood Cliffs, NJ,
  1980.

\bibitem{KSH}
{\sc T.~Kailath, A.~H. Sayed, and B.~Hassibi}, {\em Linear Estimation},
  Prentice Hall, Upper Saddle River, NJ, 2000.

\bibitem{KB}
{\sc R.~E. Kalman and R.~S. Bucy}, {\em New results in filtering and prediction
  theory}, Trans. ASME, Series D, J. Basic Eng., 83 (1961), pp.~95--107.

\bibitem{KP}
{\sc S.~W. Kim and P.~G. Park}, {\em Matrix bounds of the discrete {ARE}
  solution}, Systems Control Letters, 36 (1999), pp.~15--20.

\bibitem{LaL1}
{\sc J.~Lawson and Y.Lim}, {\em The symplectic semigroup and {R}iccati
  differential equations}, J. Dynamical and Control Syst., 12 (2006),
  pp.~49--77.

\bibitem{LaL2}
\leavevmode\vrule height 2pt depth -1.6pt width 23pt, {\em A {B}irkhoff
  contraction formula with applications to {R}iccati equations}, {SIAM J.
  Control and Optimiz.}, 46 (2007), pp.~930--951.

\bibitem{LeL}
{\sc H.~Lee and Y.~Lim}, {\em Invariant metrics, contractions and nonlinear
  matrix equations}, Nonlinearity, 2 (2008), pp.~857--878.

\bibitem{LN}
{\sc B.~C. Levy and R.~Nikoukhah}, {\em Robust state-space filtering under
  incremental model perturbations subject to a relative entropy tolerance},
  {IEEE Trans. Automat. Control}, 58 (2013), pp.~682--695.

\bibitem{LYD}
{\sc A.-P. Liao, G.~Yao, and X.-F. Duan}, {\em Thompson metric method for
  solving a class of nonlinear matrix equations}, Applied Math. and
  Computation, 216 (2010), pp.~1831--1836.

\bibitem{SDJ}
{\sc J.~L. Speyer, J.~Deyst, and D.~H. Jacobson}, {\em Optimization of
  stochastic linear systems with additive measurement and process noise using
  exponential performance criteria}, {IEEE Trans. Automat. Control}, 19 (1974),
  pp.~358--366.

\bibitem{SFB}
{\sc J.~L. Speyer, C.-H. Fan, and R.~N. Banavar}, {\em Optimal stochastic
  estimation with exponential cost crireria}, in Proc. 31st IEEE Conf. Decision
  Control, Tucson, AZ, Dec. 1992, pp.~2293--2298.

\bibitem{Whi}
{\sc P.~Whittle}, {\em Risk-sensitive Optimal Control}, J. Wiley, Chichester,
  England, 1980.

\end{thebibliography}

\end{document}